\documentclass[12pt,a4paper]{article}
\usepackage[utf8]{inputenc}
\usepackage[english]{babel}
\usepackage{amsmath}
\usepackage{amsfonts}
\usepackage{graphicx}
\usepackage{fouridx}
\usepackage{amsthm}
\usepackage{amssymb, amscd}
\usepackage{enumerate}
\usepackage{mathtools}
\usepackage[french]{varioref}
\usepackage{multirow} 
\usepackage{hyperref}
\usepackage{bm}
\usepackage{multicol}
\usepackage{multirow}
\author{ Saadaoui Nejib \thanks{ Université de Gabès, Laboratoire
		Mathématiques et Applications. Faculté des Sciences de Gabès
		Cité Erriadh 6072 Zrig Gabès Tunisie.
,\qquad \textbf{najibsaadaoui@yahoo.fr}}}
\title{ Classification of  multiplicative simple
	BiHom-Lie algebras}
\newtheorem{theorem}{Theorem}[section]

\newtheorem{prop}[theorem]{Proposition}

\newtheorem{defn}[theorem]{Definition}

\numberwithin{equation}{section}
\usepackage{multicol}
\input{epsf.sty}

\usepackage{varioref}

\newcommand{\N}{\mathbb{N}}
\newcommand{\C}{\mathbb{C}}

\newcommand{\K}{\mathbb{K}}

\begin{document} 
	\maketitle

\begin{abstract}
In this paper, we classify flnite dimensional multiplicative simple
BiHom-Lie algebras by investigating the corresponding semisimple Lie algebras and we give a complete classification of the complex
 3-dimensional
multiplicative BiHom-Lie algebras.
\end{abstract}
\section*{Introduction}
The motivations to study Hom-Lie structures are related to physics and to deformations
of Lie algebras. 
Hom-Lie algebra, introduced by
Hartwig, Larson and Silvestrov in \cite{Hartwig}, is a triple $ \left(L,[\cdot,\cdot],\alpha \right)  $ consisting of a vector space $ L $,
a bilinear map
 $ [\cdot,\cdot] \colon L\times L\to L $ 
 and a linear map $\alpha\colon L\to L  $  satisfying
the following conditions
\begin{align*}
& [x,y]=-[y,x],     &\text{   (skew-symmetry)     } \\
 &\left[ \alpha(x),[y,z]\right] +\left[\alpha(y), [z,x]\right] + \left[ \alpha(z),[x,y]\right] =0,&         \text{(Hom  Jacobi identity) }                              
\end{align*}
for all $ x,\,y,\,z\in L $.\\

In \cite{bihomM}, the authors introduced a generalized algebraic structure endowed with two commuting multiplicative linear maps, called Bihom-algebras. When the two linear maps are
same, then Bihom-algebras will be return to Hom-algebras.\\

The simple Lie algebras have been completely classified by Cartan. They fall into four
classical simple Lie algebras $ A_{n},\, B_{n},\,C_n,\,D_n $ and 
There are also five exceptional lie algebras denoted $ G_2,\, F_4,\, E_6,\, E7,\, E_8 $ which have
dimension $ 14, 52, 78, 133 $ and $ 248 $ respectively  (see \cite{simpleLie3,simpleLie1,SimpleLie2}, for example).\\

In recent years, many important results on simple Hom-Lie algebras
have been obtained (see \cite{Jin Qhomsimple1,bojemarepHomSimp} for example). In particular, 
Xue Chen and Wei Han classify finite dimensional multiplicative simple Hom-Lie algebras
by investigating the corresponding semisimple Lie algebras (see \cite{classHomSXue}).\\

It is well known that simple Lie algebras and (Hom-)Lie algebras plays an important role in (Hom-)Lie
theory.  
Similarly, it is very necessary to study simple BiHom-Lie algebras in BiHom-Lie
theory. In this paper, we classify finite dimensional multiplicative simple  BiHom-Lie algebras.\\

The paper is organized as follows. In Section 1, we recall some basic definitons about BiHom-Lie algebras. In section 2, we give the necessary and sufficient conditions for two finite dimensional multiplicative simple  BiHom-Lie algebras to be isomorphic by studying the correspending semisimple Lie algebras, and then classify finite dimensional  multiplicative simple  BiHom-Lie algebras.
In section 3,  we apply the result of the previous section to determine the  multiplicative simple  BiHom-Lie algebras of dimension $3 $.\\

  Throughout this paper, all algebras are flnite dimensional and defined on
the algebraically closed fleld $ \C $ of characteristic $ 0 $ unless otherwise specifled.

\section{Preliminaries}
\begin{defn}\cite{bihomM,Yongsheng}
	A BiHom-Lie algebra over a field $ \K $ is a $ 4 $-tuple $ \left( L,[\cdot,\cdot],\alpha,\beta \right)  $, where $  L $ is
	a  $ \K $-linear space, $ \alpha\colon L\to L $ , $ \beta \colon L\to L$  and 
	$ [\cdot,\cdot] \colon L\times L\to L$ 
	 are linear maps, satisfying the following conditions, for all $ x,\,y,\, z\in L $:
\begin{align}
&\alpha\circ \beta=\beta\circ \alpha, \\
&\alpha\left( [x,y] \right) =[\alpha(x),\alpha(y)]  \quad  \text{  and   }   \quad   \beta\left( [x,y] \right) =[\beta(x),\beta(y)],\label{multiplicative}\\
&[\beta(x),\alpha(y)]=-[\beta(y),\alpha(x)]    \text{  (skew-symmetry),}\\
 &\left[ \beta^{2}(x),[\beta(y),\alpha(z)]\right] +\left[\beta^{2}(y), [\beta(z),\alpha(x)]\right] + \left[ \beta^{2}(z),[\beta(x),\alpha(y)]\right] =0 \\
 &\qquad\qquad\text{  (BiHom-Jacobi condition)          .}   \nonumber       
\end{align}	 
 A morphism $ f\colon \left(L_{1},[\cdot,\cdot]_{1},\alpha_{1},\beta_{1} \right)\to  \left(L_{2},[\cdot,\cdot]_{2},\alpha_{2},\beta_{2} \right)  $
 of BiHom-Lie algebras is a linear map $ f\colon L_{1}\to  L_2  $
such that $ \alpha_{1}\circ f=f\circ \alpha_{2}   $,  $ \beta_{1}\circ f=f\circ \beta_{2}   $
 and $ f\left( [x,y]_{1} \right) =\left[ f(x),f(y)\right]_{2} $, for all $ x,\,y\in L_{1} $.
In particular,  BiHom-Lie algebras $ \left(L_{1},[\cdot,\cdot]_{1},\alpha_{1},\beta_{1} \right) $
and $ \left(L_{2},[\cdot,\cdot]_{2},\alpha_{2},\beta_{2} \right)  $ are isomorphic if $ f $ is an isomorphism map.\\
 A  Bihom-Lie algebra is called a regular BiHom-Lie algebra if $ \alpha,\, \beta $ are
bijective maps. 	  
\end{defn}
\begin{defn}\cite{belhsine,Yongsheng}
Let $ \left(L,[\cdot,\cdot],\alpha,\beta \right) $ be a BiHom-Lie algebra. A subspace $ \mathfrak{h} $ of $ L $	is called a BiHom-Lie subalgebra   of  $ \left(L,[\cdot,\cdot],\alpha,\beta \right) $ if $ \alpha(\mathfrak{h})\subseteq \mathfrak{h} $,    $ \beta(\mathfrak{h})\subseteq \mathfrak{h} $  and $ [\mathfrak{h},\mathfrak{h}] \subseteq \mathfrak{h}.$  In particular, a BiHom-Lie subalgebra $\mathfrak{h}  $  is said to be an ideal of $\left(L,[\cdot,\cdot],\alpha,\beta \right)   $ if $ [\mathfrak{h},L] \subseteq \mathfrak{h}$.                                  
\end{defn}
\begin{defn}
	Let $ \left(L,[\cdot,\cdot],\alpha,\beta \right) $ be a BiHom-Lie algebra.   $ \left(L,[\cdot,\cdot],\alpha,\beta \right) $ is  called a simple  BiHom-Lie algebra if $ \left(L,[\cdot,\cdot],\alpha,\beta \right) $ has no proper ideals and is not abelian. $ \left(L,[\cdot,\cdot],\alpha,\beta \right) $ is called a semisimple  BiHom-Lie algebra if $ L $     is a direct sum of certain ideals.             
\end{defn}
\begin{prop}\label{regular}
	Any finite-dimensional  simple BiHom-Lie algebra is regular.                                            
\end{prop}
\begin{proof}
	 Let $ \left(L,[\cdot,\cdot],\alpha,\beta \right) $ be a finite-dimensional  simple BiHom-Lie algebra. With $\alpha\left( [x,y] \right) =[\alpha(x),\alpha(y)]  $ and     $ \beta\left( [x,y] \right) =[\beta(x),\beta(y)]  $
we can show  that $ \ker(\alpha) $ and  $ \ker(\beta) $  are ideals of $ \left(L,[\cdot,\cdot],\alpha,\beta \right) $.
\end{proof}
\section{ Classification of multiplicative simple BiHom-Lie algebras}
	 Let $ \left(L,[\cdot,\cdot],\alpha,\beta \right) $ be a finite-dimensional  simple BiHom-Lie algebra. By Proposition \ref{regular}, $ \alpha $ and $ \beta $ are an automorphism of $ \left(L,[\cdot,\cdot],\alpha,\beta \right) $.
\begin{prop}\cite{bihomM}\label{bihom}
 Let $ \left(L,[\cdot,\cdot]'\right) $ be an ordinary Lie algebra over a field $ \C $
 and let $ \alpha,\,\beta \colon L\to L $ two commuting isomorphism  maps such that $ \beta\left([x,y]' \right)=[\beta(x),\beta(y)]'  $ for all $ x,\, y\in L $. Define the bilinear map $ [\cdot,\cdot]\colon L\times L\to L $, $ [x,y]=[\alpha(x),\beta(y)]'  $. Then 
  $ \left(L,[\cdot,\cdot],\alpha,\beta \right)  $ is a    regular BiHom-Lie algebra.                                                                                                                                                                                                                
\end{prop}
\begin{prop}\label{induced}
	Let $ \left(L,[\cdot,\cdot],\alpha,\beta \right)  $ be a regular BiHom-Lie algebra                                                                 .  Define the bilinear map $[\cdot,\cdot]'\colon L\times L\to L   $ by 
	\[ [x,y]'=[\alpha^{-1}(x),\beta^{-1}(y)],                           \]
	for all $ x,\, y\in L. $ Then $ \left(L,[\cdot,\cdot]'\right) $ is                                                                                                                                                                                                                                                                                                             a Lie algebra and
	$ \alpha $ and $ \beta $ its  also Lie  algebra automorphism. 
\end{prop}
\begin{defn} 
	 The Lie algebra  $ \left(L,[\cdot,\cdot]'\right) $   is called the induced Lie algebra of  $ \left(L,[\cdot,\cdot],\alpha,\beta \right)  $.                        
\end{defn}
 \begin{prop}
 	The two multiplicative simple BiHom-Lie algebras $ \left(L_1,[\cdot,\cdot]_1,\alpha_1,\beta_1 \right)    $ and $ \left(L_2,[\cdot,\cdot]_2,\alpha_2,\beta_{2} \right)  $ are isomorphic if and only if there exists a Lie algebra isomorphism $ f\colon (L_1,[\cdot,\cdot]_1')\to (L_2,[\cdot,\cdot]_2')  $ satisfying $ f\circ \alpha_1=\alpha_2 \circ f $	and $ f\circ \beta_1=\beta_2 \circ f $.
 \end{prop}
 \begin{proof}
 	Let $(\mathcal{G}_1,[\cdot,\cdot]'_{1})$ and $(\mathcal{G}_2,[\cdot,\cdot]'_{2})$  be induced Lie algebras of $ (\mathcal{G}_1,[\cdot,\cdot]_{1},\alpha_{1},\beta_{1}) $ and 
 	$ (\mathcal{G}_2,[\cdot,\cdot]_{2},\alpha_{2},\beta_{2}) $, respectively. Suppose $f\colon (\mathcal{G}_1,[\cdot,\cdot]_{1},\alpha_{1},\beta_{1})\to (\mathcal{G}_2,[\cdot,\cdot]_{2},\alpha_{2},\beta_{2}) $ is an isomorphism of BiHom-Lie algebras, then $f\circ \alpha_{1} = \alpha_{2} \circ f$and $f\circ \beta_{1} = \beta_{2} \circ f$, thus $f\circ \alpha_{1}^{-1} = \alpha_{2} ^{-1}\circ f$ and $f\circ \beta_{1}^{-1} = \beta_{2} ^{-1}\circ f$.
 	Moreover, 
 	\begin{align*}
 	f([x,y]'_{1})=f( [\alpha_{1}^{-1}(x),\beta_{1}^{-1}(y)]_{1})= [f(\alpha_{1}^{-1}(x)),f(\beta_{1}^{-1}(y))]_{2}&= [\alpha_{2} ^{-1}(f(x)) ,\beta_{2} ^{-1}(f(y))]_{2}\\
 	&= [f(x),f(y)]_{2}'.
 	\end{align*}
 	So $ f $ is an isomorphism between the two induced Lie algebras.\\
 	
 	On the other hand, if there exists an isomorphism $ f $ from the induced Lie algebras $ (\mathcal{G}_1,[\cdot,\cdot]_{1}') $ to 
 	$ (\mathcal{G}_2,[\cdot,\cdot]_{2}') $ such that $f\circ \alpha_{1} = \alpha_{2} \circ f  $ and $f\circ \beta_{1} = \beta_{2} \circ f$, then 
 	\[ f([x,y]_{1})=f ( [\alpha_{1}(x),\beta_{1} (y)]'_{1})= [f\circ \alpha_{1}(x) ,f\circ \beta_2(y)]_{2}'=   [\alpha_{2}\circ f(x),\beta_2\circ f(y)]_{2}'=  [f(x),f(y)]_{2}      . \]
 	Hence $f$ is an isomorphism between the two Lie algebras.  
 \end{proof}

\begin{prop}\label{semisimple}
	The induced Lie algebra of the multiplicative simple  BiHom-Lie algebra is semisimple. There exist simple  ideal $ L_1 $  and an integer $ m\neq 2 $ such that \[ L=L_1\oplus\alpha(L_1)\oplus\cdots \oplus\alpha^{m-1}(L_1)=L_1\oplus\beta(L_1)\oplus\cdots \oplus\beta^{m-1}(L_1)    .\]	
\end{prop}
\begin{proof}
	Suppose that $ \mathcal{G}_1\neq 0 $ is the maximal solvable ideal of $ \left(L,[\cdot,\cdot]' \right)  $. Because $\alpha(\mathcal{G}_1)  $ and  $\beta(\mathcal{G}_1)  $ are also solvable ideals of   $ \left(L,[\cdot,\cdot]' \right)  $, then $ \alpha(\mathcal{G}_1) \subseteq \mathcal{G}_1$  and $  \beta(\mathcal{G}_1) \subseteq \mathcal{G}_1$.
	Moreover, 
	\[ [\mathcal{G}_1,L]=\alpha\left([\alpha^{-1}(\mathcal{G}_1),\beta^{-1}(\beta(\alpha(L)))] \right)=\alpha\left([\mathcal{G}_1,\beta(\alpha(L))]' \right)\subseteq \alpha(\mathcal{G}_1) \subseteq \mathcal{G}_1,  \]
	so $ \mathcal{G}_1 $ is an ideal of $ \left(L,[\cdot,\cdot],\alpha,\beta \right)  $. Then $\mathcal{G}_1=L.  $	Furthermore, since $\left(L,[\cdot,\cdot],\alpha,\beta \right)  $ is a multiplicative simple BiHom-Lie algebra, clearly we have $ [\mathcal{G}_1,\mathcal{G}_1]=\mathcal{G}_1 $. It's contradiction. Hence $ \mathcal{G}_1=0 $.
	According to Lie theory, $\left(L,[\cdot,\cdot]' \right)  $ is a semisimple Lie algebra and
	 there are ideals $L_1,\cdots,L_m\subseteq L  $ (unique, up
	 to ordering) so that 
	$ L=L_1\oplus\cdots\oplus L_m  $ and  so that each $ (L_i,[\cdot,\cdot]') $ is a simple Lie algebra. Moreover,  since $ \alpha $ and $ \beta $ are automorphisms of Lie, we deduce that $ \left( \alpha^{k}(L_i)\right)_{k\in\N}  $ and $\left( \beta^{k}(L_i)\right)_{k\in\N}   $ are ideals of L and exist two permutations $ \sigma_{\alpha},\,\sigma_{\beta} $  respectively defined by 	 $\alpha(L_i)=L_{\sigma_{\alpha}(i)}  $,   $\beta(L_i)=L_{\sigma_{\beta}(i)}  $. Because of simpless of $ \left(L,[\cdot,\cdot],\alpha,\beta \right)  $, we have $ \sigma_{\alpha}=(1\,\cdots\, m-1),\,\sigma_{\beta}=(1\,i_1\,\cdots\, i_{m-1}) $ and $ \sigma_{\alpha}\neq  \sigma_{\beta}$. Hence, $  L=L_1\oplus\alpha(L_1)\oplus\cdots \oplus\alpha^{m-1}(L_1)=L_1\oplus\beta(L_1)\oplus\cdots \oplus\beta^{m-1}(L_1) .$	 
\end{proof}



In the decomposition $L=L_1\oplus\alpha(L_1)\oplus\cdots \oplus\alpha^{m-1}(L_1)  $,
it is clear that the Lie algebras $ L_1 $ and $ L_i $ are isomorphic therefore of the same type. 
On the other hand, if there exists an isomorphism $ f $ from the BiHom-Lie algebra $ (L,[\cdot,\cdot],\alpha,\beta) $ to $ (\mathcal{G}_1,[\cdot,\cdot]_{1},\alpha_1,\beta_1) $ so 
$ \alpha_{1} \in C(\alpha)=\{f^{-1}\circ \alpha \circ f \mid f\text{ is an automorhism of BiHom-Lie algebra}\} $ and $ \beta_{1} \in C(\beta)$.
\begin{prop}
	All finite dimensional multiplicative simple   BiHom-Lie algebras can be denoted as $(X,m,C(\alpha^{m},\beta^{m}))   $, where $ X $  represents the type  of the                           simple ideal of the correspending induced Lie algebra, $ m $ represents numbers of simple ideals, $ C(\alpha^{m},\beta^{m}) $ represents the sets of conjugate classes of the automorphisms $\alpha^{m}  $ and $\beta^{m}  $ of the simple Lie algebra $ X $, i.e $ C(\alpha^{m},\beta^{m})=\{f^{-1}\circ \alpha^{m} \circ f,\, f^{-1}\circ \beta^{m} \circ f \mid f\in Aut(X)\} $.
\end{prop}
\section{Classification of $ 3 $-dimensional 
	multiplicative simple BiHom-Lie algebras 
}
 Let $ (L,[\cdot,\cdot],\alpha,\beta) $ be a multiplicative simple BiHom-Lie algebra. Then we have the following case:

$ (L,[\cdot,\cdot],\alpha,\beta) $ is of type 
\begin{enumerate}[(1)]
\item 
$ (A_l,m,C(\alpha^{m},\beta^{m})) $ and 
 $ \dim L=ml(l+2) $.
\item 
$ (B_l,m,C(\alpha^{m},\beta^{m})) $, $ l\geq 2 $ and
 $ \dim L=ml(2n+1) $.
 \item 
$ (C_l,m,C(\alpha^{m},\beta^{m})) $,  $ l\geq 3 $ and    
 $ \dim L=ml(2l+1) $.
\item  
$ (D_l,m,C(\alpha^{m},\beta^{m})) $ 
 and $ \dim L=ml(2l-1) $.
 \item 
$ (G_2,m,C(\alpha^{m},\beta^{m})) $ 
and $ \dim L=14\times m $.
\item 
$ (F_4,m,C(\alpha^{m},\beta^{m})) $ 
 and $ \dim L=52\times m $.
\item 
$ (F_6,m,C(\alpha^{m},\beta^{m})) $ 
 and $ \dim L=78\times m $.
\item 
$ (E_7,m,C(\alpha^{m},\beta^{m})) $ 
 and $ \dim L=133\times m $.
 \item 
$ (E_8,m,C(\alpha^{m},\beta^{m})) $ and  
 $ \dim L=248\times m $.     
\end{enumerate}

Then, if  $ \dim(L)=3 $, we have $(L,[\cdot,\cdot],\alpha,\beta)  $ is of type  $ (A_1,1,C(\alpha,\beta)) $.
\begin{prop}\label{A1}
Every $ 3 $-dimensional BiHom-Lie algebra is isomorphic to one of the following nonisomorphic
BiHom-Lie algebra:\\	
$ \mathcal{L}_1 $:  $ \alpha=	 \begin{pmatrix}
		1&0&0\\
		0&a&0\\
		0&0&\frac{1}{a}
	\end{pmatrix},$  $ \beta=	 \begin{pmatrix}
	1&0&0\\
	0&b&0\\
	0&0&\frac{1}{b}
	\end{pmatrix},$  
\begin{align*}
	 [e_{1},e_{1}]&=0,&     [e_{1},e_{2}]&=2\,b\,e_2,&      [e_{1},e_3]&=-\frac{2}{b}\,e_3,\\ 
[e_2,e_{1}]&=-2\, a\,e_2,&       [e_2,e_{2}]&=0,&   [e_2,e_3]&=\frac{a}{b}\,e_1,\\
[e_3,e_{1}]&=\frac{2}{a}\,e_3,&  [e_3,e_{2}]&=-\frac{b}{a}\,e_1,&  [e_3,e_3]&=0.	
	\end{align*}
	 $ \mathcal{L}_2 $:	 $ \alpha=	 \begin{pmatrix}
	1&0&0\\
	0&1&0\\
	0&0&1
	\end{pmatrix},$  $ \beta=	 \begin{pmatrix}
	1&1&0\\
	0&1&1\\
	0&0&1
	\end{pmatrix},$  
	\begin{align*}
	[e_{1},e_{1}]&=0,&     [e_{1},e_{2}]&=2\,e_1,\\     [e_{1},e_3]&=e_1+2\,e_2,& 
	[e_2,e_{1}]&=-2\, e_1,\\       [e_2,e_{2}]&=-2\,e_1,&   [e_2,e_3]&=e_1+e_2+2\,e_3,\\
	[e_3,e_{1}]&=e_1-2\,e_2,&  [e_3,e_{2}]&= -3\,e_2-2\,e_3 ,\\
	  [e_3,e_3]&=-e_1-e_2-2\,e_3 .	
	\end{align*}
$ \mathcal{L}_3 $:	 $ \alpha=	 \begin{pmatrix}
1&1&0\\
0&1&1\\
0&0&1
\end{pmatrix},$  $ \beta=	 \begin{pmatrix}
1&a&\frac{a^2-a}{2}\\
0&1&a\\
0&0&1
\end{pmatrix},$ 
\begin{align*}
[e_{1},e_{1}]&=0,&     [e_{1},e_{2}]&=2\,e_1,& 
     [e_{1},e_3]&=(2a-1)\,e_1+2\,e_2,\\ 
[e_2,e_{1}]&=-2\, e_1,&
       [e_2,e_{2}]&=2(1-a)\,e_1,&   [e_2,e_3]&=\frac{3a-a^2}{2}e_1+3\,e_2+2\,e_3,
       \end{align*}  
\begin{align*}     
[e_3,e_{1}]&=-e_1-2\,e_2,\\ [e_3,e_{2}]&=-(a+1)\,e_1-(1+2a)\,e_2-2\,e_3,\\  [e_3,e_3]&=\frac{(1-a)(a+4)}{2}e_1+(1-a^2)\,e_2+2(1-a)\,e_3.	
\end{align*}
\end{prop}
\begin{proof}
	
	Let $ (L,[\cdot,\cdot],\alpha,\beta) $  be a  $ 3 $-dimensional  simple multiplicative  BiHom-Lie algebra. 
Then,	By Propsition \ref{semisimple}  the induced Lie algebra $ (L,[\cdot,\cdot]') $ is  $ 3 $-dimensional semisimple. Then there exists a basis $ (h,e,f) $ of $ L $ such that
   $[h,e]'=2\,e $,  $[h,f]'=-2\,e$ and  $[e,f]'=h    $.
  By Proposition \ref{induced}, $ \alpha $
   is an automorphism of Lie algebra. Hence, 
  \begin{equation}\label{multiplicative}
      [\alpha(h),\alpha(e)]'=2\alpha(e) \text{  and    }         [\alpha(h),\alpha(f)]'=-2\alpha(f).                           
  \end{equation}
Then  $ \alpha(e) $, $ \alpha(f) $ are eigenvectors    of linear map $ g=ad(\alpha(h))=[\alpha(h),\cdot] $   
 with corresponding eigenvalues $ 2 $ and  $ -2 $ respectively.
  A straightforward calculation shows that the eigenvalues of $ \alpha $ are  $ 1 $, $ a $  and $ \frac{1}{a} $. Then a basis $ (u_{1}^{i},u_{2}^{i},u_{3}^{i}) $ can be chosen for $ L $ with respect to which we obtain a matrix  of one of following form:\\
 $ \alpha_1=	 \begin{pmatrix}
 1&0&0\\
 0&a&0\\
 0&0&\frac{1}{a}
 \end{pmatrix}$ $ (a\notin \{-1,1\}) $, $ \alpha_{2}=I_3 $,
 $ \alpha_{3}=	 \begin{pmatrix}
 1&1&0\\
 0&1&1\\
 0&0&1
 \end{pmatrix},$

$ \alpha_{4}=	 \begin{pmatrix}
	1&0&0\\
	0&1&1\\
	0&0&1
\end{pmatrix},$ 
$ \alpha_{5}=	 \begin{pmatrix}
1&0&0\\
0&-1&0\\
0&0&-1
\end{pmatrix},$ 
 $ \alpha_{6}=	 \begin{pmatrix}
 1&0&0\\
 0&-1&1\\
 0&0&-1
 \end{pmatrix}.$\\
 Using $ \alpha_{i}([u_{p}^{i},u_{q}^{i}]_{i}')=[\alpha_{i}(u_{p}^{i}),\alpha_{i}(u_{q}^{i})]_{i}' $, the Jacobie identity and the induced Lie algebra is simple, we obtain:
\begin{align*}
  [u_{1}^{1},u_{2}^{1}]'&=xu_{2}^{1}, &   [u_{1}^{1},u_{3}^{1}]'&=-xu_{3}^{1} ,&   [u_{2}^{1},u_{3}^{1}]'&=yu_{1}^{1}, \\ 
  [u_{1}^{3},u_{2}^{3}]'&=xu_{1}^{3}, &   [u_{1}^{3},u_{3}^{3}]'&=-\frac{x}{2}u_{1}^{3}+xu_{2}^{3} ,&   [u_{2}^{3},u_{3}^{3}]'&=yu_{1}^{3}+\frac{x}{2}u_{2}^{3}+x u_{3}^{3}, \\ 
  [u_{1}^{5},u_{2}^{5}]'&=xu_{2}^{5}+yu_{3}^{5}, &   [u_{1}^{5},u_{3}^{5}]'&=zu_{2}^{5}-xu_{3}^{5} ,&   [u_{2}^{5},u_{3}^{5}]'&=tu_{1}^{5}.          
\end{align*} 
The Lie algebra with the bracket given by $ \alpha_{4} $ or $ \alpha_{6} $ is not simple. For the other cases $ \beta_i $ $ (i\in \{1,2,3,5\}) $ satisfies 
\begin{equation*}
	\alpha_{i}\circ \beta_{i}=\beta_{i}\circ\alpha_i \text{    and   }  \beta_{i}\left([u_{p}^{i},u_{q}^{i}]' \right)=\left[ \beta_{i}(u_{p}^{i}),\beta_{i}(u_{q}^{i})\right]' .
\end{equation*}
We obtain  $ \beta_{1}=	 \begin{pmatrix}
	1&0&0\\
	0&b&0\\
	0&0&\frac{1}{b}
\end{pmatrix}, $ $ \beta_{2}\in\left\lbrace 	 \begin{pmatrix}
1&0&0\\
0&b&0\\
0&0&\frac{1}{b}
\end{pmatrix},   \begin{pmatrix}
1&1&0\\
0&1&1\\
0&0&1
\end{pmatrix},  \begin{pmatrix}
1&0&0\\
0&1&1\\
0&0&1
\end{pmatrix}, \begin{pmatrix}
1&0&0\\
0&-1&1\\
0&0&-1
\end{pmatrix}\right\rbrace  $, $ \beta_{3}= 	 \begin{pmatrix}
1&b&\frac{b^2-b}{2}\\
0&1&b\\
0&0&1
\end{pmatrix},$ $ \beta_{5}\in\left\lbrace  \begin{pmatrix}
-1&0&0\\
0&0&\frac{1}{b}\\
0&b&0
\end{pmatrix} , \begin{pmatrix}
1&0&0\\
0&b&0\\
0&0&\frac{1}{b}
\end{pmatrix} \right\rbrace  $.\\
  Thus,  by some isomorphism of Lie algebras and $ \left[ u_{p}^{i},u_{q}^{i}\right]  =\left[ \alpha_{i}(u_{p}^{i}),\beta_{i}(u_{q}^{i})\right]' $ we we get the BiHom-Lie algebras $ \mathcal{L}_1 $, $ \mathcal{L}_2 $ and $ \mathcal{L}_3 $.                                              
\end{proof}
We note that two simple multiplicative BiHom-Lie algebras of the same type are not necessarily isomorphic.\\



\begin{thebibliography}{99}	
	\bibitem{belhsine}
	Abdelkader, BH.,\emph{Generalized Derivations of BiHom-Lie Algebras.}
	Journal of Generalized Lie
	Theory and Applications.
	\bm{$ 11 $}, $ 2011 $
\bibitem{bojemarepHomSimp}
Boujemaa, A.,  Benali, K.,  Makhlouf, A.,
\emph{  Representations of Simple Hom-Lie algebras .}
arXiv:$ 1903.08874 $v$ 1 $




	
	
	
	
\bibitem{classHomSXue}
Chen, X., Han, W.,
\emph{Classification of multiplicative simple Hom-Lie algebras. }
\bm{ $ 26 $} $ (2016) $ $ 767 $-$ 775 $.

	
	

	
	

	
	
	
	




	
	
	
	

	


	
	
	
	
	
	
	
	
	

	
	\bibitem{bihomM}
	Graziani, G.,  Makhlouf, A.,Menini,   C., and  Panaite, F., \emph{BiHom-Associative Algebras, BiHom-Lie Algebras
		and BiHom-Bialgebras }.
	Symmetry, Integrability and Geometry: Methods and Applications.
	\bm{$ 11 $} $(2015)  $,  $ 086 $, $ 34 $ pages.

	
	
	
	

.

	
	

	\bibitem{Hartwig}
Hartwig J.T., Larsson D., Silvestrov S.D.,	
\emph{Deformations of Lie algebras using
	$ \sigma $-derivations.                     }
Journal of Algebra.
\bm{ $  296 $}	
$ 2006 $, $ 314-361 $.		
	.
	\bibitem{simpleLie3}
Humphreys, 	J.,
\emph{Introduction to Lie algebras and representation theory.	  }		
 Springer-Verlag, New York, $  1972 $.
	
	
	
	
	
	
\bibitem{Jin Qhomsimple1}
Jin, Q., Li, X.,\emph{ Hom-structures on semi-simple Lie algebras} Journal of Algebra. \bm{$ 319 $}
$ (2008) $
  $ 1398-1408	 $.
	
	
	
	
	
	
	
	
	
	
	

	
	
	
	
	
	
\bibitem{simpleLie1}	
Ludwig, P.,
\emph{Algebraic Foundations of Non-Commutative Differential Geometry and Quantum Groups.	}
Springer, Berlin, Heidelberg.
\bm{$ 39 $}	$ 1996 $,
$ 5-41 $
	
	
	

	
	
	
	



 
	
	
	


	
	
	
	
	
	
	
	
	\bibitem{SimpleLie2}
Moody, R. ,
\emph{Lie algebras associated with generalized Cartan matrices.}
Advenced series in mathematical physics.
\bm{$ 3 $} $ (1967) $	
	$ 217-221 $.
	\bibitem{conformallivre}
Di Francesco, P.,
Mathieu, P.,
S\'{e}n\'{e}chal, D.,
\emph{Conformal Field Theory}
Springer, New York.
$ (1997) $. 
	
	
	
		

	
	
	
	
	
	


	
	
	
	
	
	
	


	.	

	
	
	
	
	
	
	
	
	
	
	
	
	
	
	
	
	
	
	
	
	
	
	
	
	
\bibitem{Yongsheng}
Yongsheng, C., Huange, Qi.
\emph{Representations of Bihom-Lie algebras}.
arXiv preprint arXiv:$1610.04302$, $(2016)$ 	



	
	
	
	
	
	
	
	
	
	
	
	
	
\end{thebibliography}
\end{document}